\documentclass[a4paper,11pt,fleqn]{article}

\usepackage{amsmath,amssymb,amsthm}

\allowdisplaybreaks

\newcommand{\todo}[1][\null]{\ensuremath{\clubsuit}}

\newcommand{\noprint}[1]{}

\newcommand{\Id}{\mathrm{Id}}

\newcommand{\lsemioplus}{\mathbin{\mbox{$\lefteqn{\hspace{.77ex}\rule{.4pt}{1.2ex}}{\in}$}}}

\newtheorem{theorem}{Theorem}
\newtheorem{theorem*}{Theorem}
\newtheorem{lemma}{Lemma}
\newtheorem{corollary}{Corollary}
\newtheorem{proposition}{Proposition}
{\theoremstyle{definition} 
\newtheorem{definition}{Definition}
\newtheorem{definition*}{Definition}
\newtheorem{example}{Example}

}

\begin{document}

\par\noindent {\LARGE\bf
Lie-orthogonal operators\par}

{\vspace{4mm}\par\noindent 
Dmytro R. POPOVYCH
\par\vspace{2mm}\par}

{\vspace{2mm}\par\it
\noindent Faculty of Mechanics and Mathematics, National Taras Shevchenko University of Kyiv,\\
building 7, 2, Academician Glushkov prospectus, Kyiv, Ukraine, 03127
\par}

{\vspace{2mm}\par\noindent\rm E-mail: \it  deviuss@gmail.com
 \par}

{\vspace{7mm}\par\noindent\hspace*{5mm}\parbox{150mm}{\small
Basic properties of Lie-orthogonal operators on a finite-dimensional Lie algebra are studied.
In particular, the center, the radical and the components of the ascending central series prove to be invariant with respect
to any Lie-orthogonal operator. 
Over an algebraically closed field of characteristic 0, 
only solvable Lie algebras with solvability degree not greater than two admit 
Lie-orthogonal operators whose all eigenvalues differ from~$1$ and~$-1$.
The main result of the paper is that Lie-orthogonal operators on a simple Lie algebra are exhausted by the trivial ones.
This allows us to give the complete description of Lie-orthogonal operators for semi-simple and reductive algebras,
as well as a preliminary description of Lie-orthogonal operators on Lie algebras with nontrivial Levi--Mal'tsev decomposition.
The sets of Lie-orthogonal operators of some classes of Lie algebras (Heisenberg algebras, almost Abelian algebras, etc.)
are directly computed. 
In particular, it appears that the group formed by the equivalence classes of Lie-orthogonal operators on a Heisenberg algebra is isomorphic 
to the standard symplectic group of an appropriate dimension. 
}\par\vspace{5mm}}


\vspace{2mm}


\section{Introduction}

The condition of Lie orthogonality of operators defined on Lie algebras naturally arises in the course of the study of certain structures connected to Lie algebras
such as K\"ahler manifolds and Clifford structures~\cite{Barberis&Dotti2004, Barberis&Dotti&Miatello1995, Dotti&Fino2000, Grantcharov&Poon2000}.
The simplest among these structures are Abelian complex structures on real Lie algebras~\cite{Barberis&Dotti2004, Barberis&Dotti&Miatello1995, Dotti&Fino2000},
i.e., linear operators satisfying the following two conditions:
\begin{gather*}
1)\ J^2=-\Id, \\
2)\ [Jx,Jy]=[x,y].
\end{gather*}
Lie algebras carrying such structures are exhaustively classified only up to dimension six~\cite{Andrada&Barberis&Dotti2011}. 

The object studied in the present paper is the class of operators satisfying the second condition alone, which is called the condition of Lie orthogonality.
The renunciation of the first condition helps us to understand what properties of Abelian complex structures are implied by the second condition.
For example, Proposition~3.1 from~\cite{Barberis&Dotti2004} directly follows from Theorem~\ref{TheoremOnSolvabilityOfLieAlgWithLieOrtOp} of the present paper.

Moreover, Lie-orthogonal operators are also related to a particular case of so-called generalized differentiations in the same way as
automorphisms are related to usual differentiations. Recall that a linear operator~$D\in \mathrm{End}(L)$ is called 
a generalized $(\alpha, \beta, \gamma)$-differentiation on a Lie algebra~$L$ if for fixed constants $\alpha$,~$\beta$ and~$\gamma$
the equation $\alpha D[x,y]=\beta[Dx,y]+\gamma[x, Dy]$ holds for any elements~$x$ and~$y$ from the algebra~$L$~\cite{Novotny&Hrivnyak2008}.
Lie-orthogonal operators are associated with $(0,1,1)$-differentiations.

In this paper we enhance and generalize results from~\cite{Bilun&Maksimenko&Petravchuk2011,Petravchuk&Bilun2003} 
and completely describe the sets of Lie-orthogonal operators for certain classes of Lie algebras.

The paper has the following structure: 
The simplest notions and properties related to Lie-orthogonality are considered in Section~\ref{SectionOnElemPropsLieOrtOps}. 
Elementary algebraic structures on sets of such operators are also discussed.
Actions of Lie-orthogonal operators on the center and the radical of a Lie algebra
as well as on ideals generated by the corresponding generalized eigenspaces and the center are studied
in Section~\ref{SectionOnInvIdealsOfLieOrtOps}. 
In the same section we introduce the notion of equivalence of Lie-orthogonal operators,
which is especially important in the course of the consideration of Lie algebras with nonzero centers.
Lie-orthogonal automorphisms are also studied.
It is proved in Section~\ref{SectionOnLieOrtOpsOnNonzeroLeviFactor} that any simple Lie algebra
possesses only the trivial Lie-orthogonal operators. 
Using this fact, we completely describe Lie-orthogonal operators on semi-simple and reductive Lie algebras
and derive some results for general Lie algebras.
The basis approach is applied in Section~\ref{SectionOnDirectCalcOfLieOrtOps} to the direct calculation
of the sets of Lie-orthogonal operators for important classes of Lie algebras 
including the special linear algebras, Heisenberg algebras, the almost Abelian algebras, etc.

\section{Elementary properties of Lie-orthogonal operators}\label{SectionOnElemPropsLieOrtOps}

In this section we consider an arbitrary Lie algebra $L$ 
as the imposed restrictions on the dimension and the underlying field are not essential.

\subsection{Definitions and examples}

\begin{definition}
A linear operator~$J$ on~$L$ is called \emph{Lie-orthogonal} if $[Jx, Jy]=[x, y]$ for any $x, y \in L$. 
\end{definition}

\begin{example}
Each linear operator on an Abelian Lie algebra is Lie-orthogonal.
The zero operator is Lie-orthogonal if and only if the corresponding Lie algebra is Abelian.
\end{example}

\begin{example}
For any Lie algebra $L$, the identity operator $\Id_L$ as well as $-\Id_L$ are Lie-orthogonal over~$L$.
We will call these operators \emph{trivial} Lie-orthogonal operators on~$L$.
Only the Lie-orthogonal operators which are different from the trivial ones are of interest.
\end{example}

For certain Lie algebras the associated sets of Lie-orthogonal operators are exhausted by the trivial operators.

\begin{example}
Consider the Lie algebra $L=\mathrm{sl}_2$. We choose a basis $\{e_1, e_2, e_3\}$ such that the associated commutation relations are
$[e_1, e_2]=e_3$, $[e_2, e_3]=e_1$, $[e_3, e_1]=e_2$, i.e.\ $[e_i, e_j]=e_k$ for an even transposition of the indices $i$, $j$, $k$.
In other words, we use the representation of the algebra $\mathrm{sl}_2$ as the algebra of three-dimensional vectors with the vector product as a Lie bracket.
Let $J$~be a Lie-orthogonal operator on $\mathrm{sl}_2$.
By definition, $e_k=[e_i, e_j]=[Je_i, Je_j]$. Hereby the vector $e_k$ is orthogonal to the vectors $Je_i$ and $Je_j$. 
Recomposing these relations, we obtain that for any $i$ the vector $Je_i$ is orthogonal to $e_j$ and $e_k$, and hence it is proportional to
$e_i$: $Je_i=\lambda_ie_i$ for some constant~$\lambda_i$.
The commutation relations imply that $\lambda_i\lambda_j=1$, $i\ne j$, from which either $\lambda_1=\lambda_2=\lambda_3=1$ or $\lambda_1=\lambda_2=\lambda_3=-1$. 
This means that either $J=\Id_L$ or $J=-\Id_L$.
Therefore, the algebra $\mathrm{sl}_2$ admits only the trivial Lie-orthogonal operators.
\end{example}

\begin{proposition}
Suppose that $L=L_1\oplus \dots \oplus L_k$, i.e., $L$ is a direct sum of its ideals $L_1$, \ldots, $L_k$, and $J_i$ is a Lie-orthogonal operator on~$L_i$, $i=1,\dots,k$.
Then the operator $J=J_1\oplus \dots \oplus J_k$ is Lie-orthogonal on~$L$.
\end{proposition}

\begin{example}
If the Lie algebra~$L$ is a direct sum of its ideals $L_1$, \ldots, $L_k$ and for each $i=1,\dots,k$ either $J_i=\Id_{L_i}$ or $J_i=-\Id_{L_i}$ 
then the operator $J=J_1\oplus \dots \oplus J_k$ is Lie-orthogonal on~$L$.
\end{example}

\subsection{Algebraic structures related to Lie-orthogonal operators}

\begin{proposition}
If $L$ is a Lie algebra, $J$ is a Lie-orthogonal operator on~$L$ and $S$~is a subalgebra of~$L$ which is invariant with respect to~$J$
then the restriction of the operator~$J$ on~$S$ is a Lie-orthogonal operator on the subalgebra~$S$.
\end{proposition}

\begin{proposition}
If $J$~is a Lie-orthogonal operator on a Lie algebra $L$ then the operator~$-J$ is also Lie-orthogonal on this algebra.
\end{proposition}

\begin{proposition}\label{PropositionOnLieOrthogonalityOfInverseOperator}
Let $J$~be a non-degenerate Lie-orthogonal operator on a Lie algebra~$L$. Then the inverse $J^{-1}$ of $J$ is also a Lie-orthogonal operator on the algebra $L$.
The condition of Lie orthogonality of the operator $J$ is equivalent to the condition $[Jx, y]=[x, J^{-1}y]$ for any elements $x, y \in L$.
\end{proposition}

\begin{proposition}\label{PropositionOnLieOrthogonalityOfComposition}
The composition of Lie-orthogonal operators on a Lie algebra $L$ is a Lie-orthogonal operator on this algebra.
Therefore, the set of Lie-orthogonal operators on~$L$ is a monoid with involution with respect to the composition of operators.
The non-degenerate Lie-orthogonal operators on~$L$ form a group with respect to the same operation.
\end{proposition}

Using arguments from the proof of Theorem~$3$ from~\cite{Bilun&Maksimenko&Petravchuk2011}, we derive the following assertion:

\begin{proposition}
Let $J$ be a Lie-orthogonal operator on a Lie algebra~$L$ and let $S$ be an automorphism of $L$.
Then the operator $S^{-1}JS$ is also Lie-orthogonal on~$L$.
\end{proposition}

\begin{proof}
For any $x$ and $y$ from the algebra~$L$ the operator $\tilde J=S^{-1}JS$ satisfies the condition of Lie orthogonality: 
$[\tilde Jx,\tilde Jy]=[S^{-1}JSx, S^{-1}JSy]=S^{-1}[JSx, JSy]=S^{-1}[Sx, Sy]=[x,y]$.
\end{proof}

\begin{corollary}
The set of Lie-orthogonal operators on a Lie algebra~$L$ is a $G$-set with $G=\mathrm{Aut}(L)$.
\end{corollary}

\subsection{Basis approach}

Consider an $n$-dimensional Lie algebra $L$ and a Lie-orthogonal operator~$J$ on~$L$.
We fix a basis $\{e_1, \dots, e_n\}$ of~$L$.
The commutation relations of~$L$ in the fixed basis take the form $[e_i, e_j]=c_{ij}^ke_k$, 
where $c_{ij}^k$ are components of the structure constant tensor of~$L$, and the matrix of the operator~$J$ is~$(J_{ij})$.
Here and in what follows the indices $i$, $j$ and $k$ run from~$1$ to~$n$, and we assume summation over repeated indices.

We write down the definitions of Lie orthogonality for each pair of basis elements and expand all involved elements of the algebra with respect to the basis:
\begin{gather*}
[Je_i, Je_j]=[J_{i'i}e_{i'}, J_{j'j}e_{j'}]=J_{i'i}J_{j'j}[e_{i'}, e_{j'}]=J_{i'i}J_{j'j}c_{i'j'}^ke_k=[e_i, e_j]=c_{ij}^ke_k. 
\end{gather*}
Hence $J_{i'i}J_{j'j}c_{i'j'}^k=c_{ij}^k$.
In the matrix notation the last condition is represented in the form
\[
J^{\mathrm T}C^kJ=C^k,
\]
where for each fixed~$k$ the matrix $C^k=(c_{ij}^k)$ is skew-symmetric. 
It appears to be the condition of the invariance of the bilinear skew-symmetric form associated with the matrix~$C_k$ with respect to the operator~$J$ for each~$k$.
This implies the system of at most $n^2(n-1)/2$ quadratic inhomogeneous equations for the coefficients of the matrix of the operator~$J$.
The solution of the system completely describes the set of Lie-orthogonal operators on the Lie algebra~$L$.
(See Section~\ref{SectionOnDirectCalcOfLieOrtOps} for examples of calculations involving the basis approach.)

\section{Invariant ideals of Lie-orthogonal operators}\label{SectionOnInvIdealsOfLieOrtOps}

In the theory of Lie-orthogonal operators an important role is played by invariant subspaces of these operators, which are also ideals of the corresponding Lie algebras.

\subsection{Center}

A special place among ideals that are invariant with respect to all Lie-orthogonal operators over a given Lie algebra~$L$ belongs to the center~$Z$ of~$L$.
The properties of Lie-orthogonal operators on~$L$ essentially depend on whether or not the center~$Z$ is zero.

\begin{lemma}[\cite{Petravchuk&Bilun2003}]\label{LemmaOnZeroFittingComponentInCenter}\looseness=-1
Let $L$ be a finite-dimensional Lie algebra and $J$~be a Lie-orthogonal operator on~$L$.
Then $L_0\subseteq Z$, where $Z$ is the center of~$L$ and $L_0$ is the generalized eigenspace of~$J$ corresponding to zero 
(i.e., the zero component of the Fitting decomposition of~$L$ with respect to~$J$).
\end{lemma}

\begin{proof} 
We prove this lemma in another way than that in~\cite{Petravchuk&Bilun2003}, without using the Jordan normal form of~$J$.
Let $k$ be the nilpotency degree of the restriction of $J$ on~$L_0$.
Given arbitrary elements $x\in L_0$ and $y\in L$, we have $J^kx=0$.
Since $J^k$ is a Lie-orthogonal operator on $L$, as $J$ is, we obtain 
$[x, y]=[J^kx, J^ky]=[0, J^ky]=0$, and hence $x$ belongs to the center~$Z$.
\end{proof}

\begin{corollary}
Any Lie-orthogonal operator on a centerless Lie algebra is invertible. 
The Lie-orthogonal operators on such an algebra form a group.
\end{corollary}

\begin{lemma}\label{LemmaOnInvCenter}
Let $L$ be a finite-dimensional Lie algebra and $J$~be a Lie-orthogonal operator on~$L$. 
Then $J(Z)\subseteq Z$, where $Z$ is the center of~$L$.
\end{lemma}

\begin{proof}
Suppose that $x\in Z$. Then $[x, y]=0$ for any $y\in L$.
It is necessary to prove that any $y\in L$ satisfies the condition $[Jx, y]=0$.
We write down the Fitting decomposition (see e.g. \cite{Jackobson1964}) of the space~$L$ with respect to the operator~$J$:
$L=L_0\dotplus \hat L$, where $L_0$~and~$\hat L$ are invariant subspaces of~$J$, 
and the restriction of~$J$ on $L_0$ (resp.\ $\hat L$) is a nilpotent (resp.\ invertible) operator. 
Here and in what follows the symbol ``$\dotplus$'' denotes the direct sum of vector spaces. 
Hence for each $y$ from $L$ we have $y=y_0+y_1$, where $y_0\in L_0$ and $y_1\in \hat L$.
Properties of $J$, the definition of $y_1$ and the bilinearity of the Lie bracket imply that
$0=[x, y]=[Jx, Jy]=[Jx, Jy_1]$. 
As the operator~$J$ is invertible on $\hat L$,  
the element $\tilde y_1=Jy_1$ runs through $\hat L$ when the element $y_1$ runs through $\hat L$.
Given an arbitrary element $\tilde y_0\in L_0$, the element $\tilde y=\tilde y_1+\tilde y_0$ runs through the entire space~$L$ and the following equality is true:
$0=[Jx, Jy_1]=[Jx, \tilde y_1]=[Jx, \tilde y_1]+[Jx, \tilde y_0]=[Jx, \tilde y]$.
\end{proof}

\begin{lemma}\label{LemmaOnIndeterminacyOfLieOrtOpsOnCenter1}
Suppose that $L$ is a finite-dimensional Lie algebra, its center~$Z$ is nonzero and
$\hat L$ is an arbitrary subspace of the space~$L$ such that $L=Z\dotplus\hat L$. 
If $J$~is a Lie-orthogonal operator on~$L$
and an operator $\tilde J$ on $L$ satisfies the conditions $\tilde J|_{\hat L}=J|_{\hat L}$ and $\tilde J(Z)\subseteq Z$
then $\tilde J$ is a Lie-orthogonal operator on~$L$. 
\end{lemma}

In other words, a Lie-orthogonal operator can be arbitrarily redefined on the center of the corresponding Lie algebra.

\begin{proof}
We represent arbitrary elements $x$ and $y$ from $L$ in the form $x=x_0+x_1$ and $y=y_0+y_1$, where $x_0, y_0 \in Z$ and $x_1, y_1\in \hat L$. Then
\begin{gather*}
[\tilde Jx, \tilde Jy]=[\tilde Jx_0+\tilde Jx_1, \tilde Jy_0+\tilde Jy_1]=[\tilde Jx_1, \tilde Jy_1]=[Jx_1, Jy_1]=\\
=[x_1, y_1]=[x_0+x_1, y_0+y_1]=[x, y].
\end{gather*}
This means that $\tilde J$ is a Lie orthogonal operator.
\end{proof}

The following assertion generalizes Lemma~\ref{LemmaOnIndeterminacyOfLieOrtOpsOnCenter1}. 

\begin{lemma}\label{LemmaOnIndeterminacyOfLieOrtOpsOnCenter2}
Suppose that $L$ is a finite-dimensional Lie algebra, $Z$ is its center,
$J$~is a Lie-orthogonal operator on~$L$ and $J_0$~is an operator on~$L$ with the image contained in~$Z$, $J_0L\subseteq Z$.
Then $J+J_0$~is a Lie-orthogonal operator on~$L$. 
\end{lemma}

\begin{proof}
We take arbitrary elements $x$ and $y$ from $L$. As $J_0x$ and $J_0y$ belong to the center by lemma's hypotheses, we have
$[(J+J_0)x, (J+J_0)y]=[Jx, Jy]+[J_0x, Jy]+[Jx, J_0y]+[J_0x, J_0y]=[Jx, Jy]=[x, y]$. This completes the proof.
\end{proof}

In Lemma~\ref{LemmaOnIndeterminacyOfLieOrtOpsOnCenter1} the operator $J_0$ is zero on~$\hat L$.
Lemma \ref{LemmaOnIndeterminacyOfLieOrtOpsOnCenter2} allows us to introduce an equivalence relation on the set of Lie-orthogonal operators on a fixed algebra~$L$.

\begin{definition}
We call Lie-orthogonal operators $J$ and $\tilde J$ on a Lie algebra \emph{equivalent} if the image of their difference is contained in the center of this algebra.
\end{definition}

The equivalence relation is consistent with the action of the automorphism group of the algebra on its set of Lie-orthogonal operators.
If the algebra is centerless, only coinciding operators are equivalent. 
Lemmas~\ref{LemmaOnZeroFittingComponentInCenter} and~\ref{LemmaOnIndeterminacyOfLieOrtOpsOnCenter1} imply the following assertion:

\begin{corollary}
Any Lie-orthogonal operator on a finite-dimensional Lie algebra is equivalent to an invertible Lie-orthogonal operator on the same algebra.
\end{corollary}

Given a decomposition of an $n$-dimensional algebra~$L$ into the direct sum of the center~$Z$ and its complement~$\hat L$ as vector spaces ($L=Z\dotplus \hat L$), 
the ``essential'' part of any Lie-orthogonal operator~$J$ on~$L$ is the operator $PJP$, 
where $P$~is the operator of projection onto~$\hat L$ in the above decomposition of~$L$. 
The operator $PJP$ is associated with the factorized operator $J/Z$ on the factor-algebra $L/Z$ (cf.\ Section~\ref{SectionOnSpecIdeals}).
After fixing a basis in such a way that the first~$k$ elements of it form a basis of~$Z$, where $k=\dim Z$, and the others form a basis of~$\hat L$,  
we obtain that the matrix of any Lie-orthogonal operator~$J$ on~$L$ in the chosen basis can be represented in the form
\[
\left(
\begin{array}{cc}
B_0&B_1\\
0&\hat J
\end{array}
\right),
\]
where $B_0$ and $B_1$~are arbitrary $k\times k$ and $k\times (n-k)$  matrices, respectively, $0$~is the zero $(n-k)\times k$ matrix,  
and $\hat J$~is the matrix of the restriction of $PJP$ on~$\hat L$.
The matrix~$\hat J$ can also be interpreted as the matrix of the factorized operator~$J/Z$ on the factor-algebra~$L/Z$.

\begin{proposition}
The equivalence classes of Lie-orthogonal operators on a finite-dimensional Lie algebra form a group.
\end{proposition}

\begin{proof}
Let $L$ be a finite-dimensional Lie algebra and let $Z$ denote the center of~$L$.
In view of Propositions~\ref{PropositionOnLieOrthogonalityOfInverseOperator} and~\ref{PropositionOnLieOrthogonalityOfComposition}
it suffices to check the correctness of the group operations defined in terms of representatives of equivalence classes. 

Suppose that $J_1$, $J_2$, $\tilde J_1$ and~$\tilde J_2$ are Lie-orthogonal operators on~$L$ and  
the operators~$\tilde J_1$ and~$\tilde J_2$ are equivalent to the operators~$J_1$ and~$J_2$, respectively. 
Then the operator~$\tilde J_1\tilde J_2$ is equivalent to the operator~$J_1J_2$. 
Indeed, $\tilde J_1\tilde J_2-J_1J_2=(\tilde J_1-J_1)\tilde J_2+J_1(\tilde J_2-J_2)$. 
As $(\tilde J_1-J_1)L\subseteq Z$, $(\tilde J_1-J_1)L\subseteq Z$ and $J_1Z\subseteq Z$, 
we have that $(\tilde J_1\tilde J_2-J_1J_2)L\subseteq Z$. 

Suppose that $J$ and~$\tilde J$ are equivalent invertible Lie-orthogonal operators on~$L$.
The equality $\tilde J^{-1}-J^{-1}=-J^{-1}(\tilde J-J)\tilde J^{-1}$ and the conditions $(\tilde J-J)L\subseteq Z$ and $J^{-1}Z\subseteq Z$
imply that $(\tilde J^{-1}-J^{-1})L\subseteq Z$, i.e., the Lie-orthogonal operators~$J^{-1}$ and~$\tilde J^{-1}$ are equivalent.
\end{proof}

\subsection{Radical and other special ideals}\label{SectionOnSpecIdeals}

\begin{lemma}\label{LemmaOnInvarienceOfRad}
Let $L$~be a finite-dimensional Lie algebra and $J$~be a Lie-orthogonal operator on~$L$. Then
$J(R)\subseteq R$, where $R=R(L)$~is the radical of the algebra~$L$, i.e., its maximal solvable ideal.
\end{lemma}
\begin{proof}
By $K(x, y)$ we denote the Killing form of the algebra~$L$. 
The radical $R$ is the orthogonal complement to the derived algebra $L'$ of the algebra $L$ with respect to the Killing form $K$~\cite{Bourbaki1975},~i.e.,
\[
R=\{x\in L\mid \forall y,z\in L\colon K(x, [y,z])=0\}.
\]
Given an arbitrary $x$ from $R$ and arbitrary $y$ and $z$ from $L$, 
the definition of Lie-orthogonal operator and the associativity of the Killing form with respect to the Lie bracket imply 
\begin{gather*}
K(Jx, [y,z])=K(Jx, [Jy, Jz])=K([Jx, Jy], Jz)=K([x, y], Jz)=K(x, [y, Jz])=0.
\end{gather*}
This means that $Jx\in R$.
\end{proof}

It is clear from the proof of Lemma 3 from \cite{Bilun&Maksimenko&Petravchuk2011} that this lemma is true not only for invertible Lie-orthogonal operators.
Therefore the lemma can be reformulated in the following way:

\begin{lemma}\label{LemmaOnGenInvIdeals}
Suppose that $I$ is an ideal of a Lie algebra $L$ such that the factor-algebra $L/I$ is centerless.
Then the ideal $I$ is invariant with respect to the action of each Lie-orthogonal operator on the algebra~$L$.
\end{lemma}

\begin{proof}
In the course of the factorization with respect to the ideal~$I$, 
elements of the center of the algebra $L$ are mapped into the center of the factor-algebra $L/I$, which is zero by the lemma's hypothesis. 
This implies that the center of the algebra is contained in the ideal $I$.
Hence this ideal is invariant with respect to a Lie-orthogonal operator~$J$ if and only if it is invariant with respect to any Lie-orthogonal operator which is equivalent to~$J$.
This is why we can assume, without loss of generality, that the operator~$J$ is invertible.

Using the reformulated definition of Lie-orthogonal operator, we obtain that for any $x\in I$ and $y\in L$ the commutator
$[Jx, y]=[x, J^{-1}y]$ belongs to the ideal~$I$. Therefore, the equivalence class $Jx+I$ belongs to the center of the factor-algebra $L/I$. This implies that $Jx\in I$. 
\end{proof}

If an ideal $I$ of a Lie algebra~$L$ is invariant with respect to the action of a Lie-orthogonal operator~$J$ on~$L$ then the operator~$J$ can be factorized consistently with factorizing the algebra~$L$.
The factor-operator~$J/I$ on the factor-algebra $L/I$ is also Lie-orthogonal.
At the same time, factorized Lie-orthogonal operators do not exhaust all possible Lie-orthogonal operators on the factor-algebra.

Iteratively combining Lemma~\ref{LemmaOnInvCenter} with factorizing with respect to the corresponding centers, we obtain a generalization of Corollary~2 from~\cite{Bilun&Maksimenko&Petravchuk2011}.

\begin{proposition}
Each element of the ascending central series of a Lie algebra~$L$ is invariant with respect to any Lie-orthogonal operator on~$L$.
\end{proposition}

\begin{lemma}\label{LemmaOnOpOnSumIdealsEqualsSumOp}
Let $L$~be a finite-dimensional centerless Lie algebra, which is decomposed into the direct sum of its ideals $I_1$, \dots $I_k$:
$L=I_1\oplus \dots \oplus I_k$. 
An operator~$J$ is Lie-orthogonal on~$L$ if and only if it can be represented as
\[
J=J_1\oplus \dots \oplus J_k,
\]
where for any $i=1,\dots,k$ the operator~$J_i$ is Lie-orthogonal on $I_i$. 
If the center of the algebra~$L$ is not zero then the same representation is true up to the equivalence relation of Lie-orthogonal operators. 
\end{lemma}

\begin{proof}
The statement of the lemma is equivalent to the fact that each of the ideals is invariant 
(up to the equivalence relation) with respect to the action of any Lie-orthogonal operator.
It suffices to prove the invariance of the ideals~$I_1$, \dots, $I_k$ for the particular case $k=2$. 
The center~$Z$ of the algebra $L=I_1\oplus I_2$ can be represented in the form $Z=Z_1\oplus Z_2$, where $Z_1=Z\cap I_1$ and $Z_2=Z\cap I_2$.
Moreover, $L/I_1\simeq I_2$ and $L/I_2\simeq I_1$.
Then, if the algebra~$L$ is centerless, the invariance of the ideals $I_1$ and $I_2$ 
directly follows from Lemma~\ref{LemmaOnGenInvIdeals}, as these ideals are also centerless.

Suppose that the center of the algebra~$L$ is not zero. Up to the equivalence relation we can consider only invertible Lie-orthogonal operators. We fix such an operator~$J$.
By $P_1$ and $P_2$ we denote the operators of the projection on the ideals~$I_1$ and~$I_2$ respectively, which are associated with the decomposition $L=I_1\oplus I_2$. We have $P_1+P_2=\Id_L$.
For arbitrary elements $x\in I_1$ and $y\in I_2$ we derive that the commutator $[Jx, y]=[x, J^{-1}y]$ belongs to the intersection of the ideals~$I_1$ and~$I_2$, therefore, it is equal to~$0$.
This means that $P_2Jx\in Z_2$ and $P_1Jy\in Z_1$. Hence the images of the operators $P_1JP_2$ and $P_2JP_1$ are contained in~$Z$.
Consider the operator $\tilde J=J-P_1JP_2-P_2JP_1$. It is equivalent to the operator~$J$ in view of the definition.
Moreover, if $x\in I_1$ then $\tilde Jx=Jx-P_2Jx=P_1Jx\in I_1$. Analogously, if $y\in I_2$ then $\tilde Jy=Jy-P_1Jy=P_2Jy\in I_2$.
\end{proof}

\subsection{Generalized eigenspaces of Lie-orthogonal operators}

Suppose that the underlying field is of characteristic zero and algebraically closed.

\begin{lemma}[\cite{Petravchuk&Bilun2003}] \label{LemmaOnCommRootSubspEq0}
Suppose that $L$ is a finite-dimensional Lie algebra, $J$ is a Lie-orthogonal operator on~$L$ and $\lambda$ and~$\mu$ are eigenvalues of~$J$ such that $\lambda\mu\ne 1$. 
If $L_\lambda$ and $L_\mu$ are the generalized eigenspaces of~$J$ corresponding to $\lambda$ and~$\mu$, respectively, then $[L_\lambda, L_\mu]=\{0\}$.
\end{lemma}

\begin{proof}
We prove this lemma without choosing a basis of the algebra $L$.
Consider eigenvalues~$\lambda$ and $\mu$ of the operator $J$ such that $\lambda\mu\ne 1$.
The corresponding generalized eigenspaces $L_\lambda$ and~$L_\mu$ can be represented in the form:
\[
L_\lambda=\bigcup_{i=0}^{k_\lambda}\ker(J-\lambda E)^i, \quad
L_\mu=\bigcup_{j=0}^{k_\mu}\ker(J-\mu E)^j,
\]
where $k_\lambda$ and $k_\mu$ denote the multiplicity of $\lambda$ and $\mu$ as the roots of the characteristic polynomial of the operator $J$ and $E$ denotes $\Id_L$.
Therefore, it suffices to prove that for any $i$ and $j$ arbitrary elements $x\in \ker(J-\lambda E)^i$ and $y\in \ker(J-\mu E)^j$ commutate.
The last assertion is proved by induction with respect to $m=i+j$.
If $m=0$ then $i=j=0$, both the elements~$x$ and~$y$ are zero as elements of the kernel of the identity operator and hence $[x, y]=0$.
Supposing that this assertion is true for all $(i,j)$ with $i+j<m$, we prove it for~$m$.
The definition of Lie orthogonality implies
\begin{align*}
[x,y]&=[Jx, Jy]=[(J-\lambda E)x, (J-\mu E)y]+[(J-\lambda E)x, \mu y]+[\lambda x, (J-\mu E)y]+[\lambda x, \mu y]\\&=\lambda\mu[x, y].
\end{align*}
The first three summands are zero by the induction hypothesis as $(J-\lambda E)x\in\ker(J-\lambda E)^{i-1}$ and $(J-\mu E)y\in\ker(J-\mu E)^{j-1}$.
Taking into account that $\lambda\mu\ne 1$, we obtain $[x, y]=0$.
\end{proof}

We generalize Lemma~3 from~\cite{Petravchuk&Bilun2003} and 
the main theorem of the same paper via getting rid of the condition that the corresponding algebra is centerless. 

\begin{lemma}\label{LemmaOnIdealsRelatedToRootSubsp}
Let $L$ be a finite-dimensional Lie algebra with the center~$Z$ and let $J$ be a Lie-orthogonal operator on~$L$. 
For any eigenvalue~$\lambda$ of the operator~$J$, we consider the subspace $I_\lambda$, where 
$I_\lambda=L_\lambda\oplus L_{\lambda^{-1}}\oplus Z$ if $\lambda\notin\{\pm 1, 0\}$,
$I_\lambda=L_\lambda\oplus Z$ if $\lambda=\pm 1$, 
$I_\lambda=Z$ if $\lambda=0$.
Then the subspace $I_\lambda$ is an ideal of the algebra~$L$.
\end{lemma}

\begin{proof}
For the case $\lambda=0$ the lemma is obvious. Suppose that $\lambda \neq 0$. 
Consider the subspace $S_\lambda=\sum_{\mu\notin\{\lambda,\lambda^{-1}\}}L_\mu$ of~$L$. 
We have $I_\lambda+S_\lambda=L$.
In view of Lemma~\ref{LemmaOnCommRootSubspEq0} the centralizer~$C_{\lambda}$ of $S_\lambda$ in~$L$ contains~$I_\lambda$.
If an element of~$C_{\lambda}$ does not belong to~$I_\lambda$ then $U=(S_\lambda\cap C_\lambda)\setminus I_\lambda\neq\varnothing$.
The elements of the set~$U$ commute with all elements of the subspace~$S_\lambda$ by the definition of centralizer
as well as with all elements of the subspace~$I_\lambda$ according to Lemma~\ref{LemmaOnCommRootSubspEq0} and the definition of center. 
Therefore, the set~$U$ is contained in the center~$Z$ which is a subset of~$I_\lambda$. At the same time, $U$ is contained in the compliment to $I_\lambda$ by its definition.
As a result, $C_{\lambda}=I_\lambda$ and hence $I_\lambda$ is a subalgebra of the algebra~$L$.
Then Lemma~\ref{LemmaOnCommRootSubspEq0} implies that $I_\lambda$ is an ideal of this algebra.
\end{proof}

\begin{corollary}
If $\lambda\neq\pm 1$, $I_\lambda$ is a solvable ideal of~$L$ of solvability degree not greater than two.
\end{corollary}
\begin{proof}
As $\lambda \neq\pm 1$, Lemma~\ref{LemmaOnCommRootSubspEq0} implies that $L_\lambda$ and $L_{\lambda^{-1}}$ are Abelian subalgebras of the algebra~$L$ and, therefore, of the ideal~$I_\lambda$.
We introduce the notation $Z_\lambda=L_\lambda\cap Z$ and choose the subspace $\tilde Z_\lambda$ of the center~$Z$ such that $Z=Z_\lambda\dotplus\tilde Z_\lambda=Z_\lambda\oplus\tilde Z_\lambda$.
Then we have $I_\lambda=L_\lambda\dotplus\tilde L_\lambda$, where $\tilde L_\lambda=L_{\lambda^{-1}}+\tilde Z_\lambda$.
As the subalgebras~$L_\lambda$ and~$\tilde L_\lambda$ are Abelian, the ideal $I_\lambda$ is a sum of two Abelian subalgebras. 
Therefore, it is solvable and its solvability degree is not greater than two (see e.g.~\cite{Kostrikin1982} or Lemma~1 in~\cite{Petravchuk1988}).
\end{proof}

\begin{corollary}\label{CorollaryEigenJOnSimpleAlg}
If $L$ is a finite-dimensional simple Lie algebra and $J$ is a Lie-orthogonal operator on~$L$ 
then all eigenvalues of~$J$ are equal to either $1$ or $-1$.
\end{corollary}

\begin{proof}
A simple algebra is centerless, is not solvable and contains no ideals. 
Therefore, the operator~$J$ possesses a single generalized eigenspace. The corresponding eigenvalue equals either $1$ or~$-1$.
\end{proof}

\begin{corollary}\label{CorollaryEigenJOnSemisimpleAlg}
If $L$ is a finite-dimensional semi-simple Lie algebra and $J$ is a Lie-orthogonal operator on~$L$ 
then each eigenvalue of the operator~$J$ is equal to either $1$ or $-1$.
Each simple component of the algebra~$L$ is contained in one of the generalized eigenspaces of~$J$.
\end{corollary}

\begin{proof}
A semi-simple algebra is centerless and contains no nonzero solvable ideals.
Therefore, the algebra~$L$ is a direct sum of two generalized eigenspaces, $L_1$ and~$L_{-1}$, 
which are also ideals and correspond to the eigenvalues~$1$ and~$-1$, respectively.
As any simple component of the algebra~$L$ is an ideal in this algebra and does not contain proper ideals, 
it either does not intersect with one of the above generalized eigenspaces or is contained in it. \end{proof}

\begin{theorem}\label{TheoremOnSolvabilityOfLieAlgWithLieOrtOp}
Let $L$ be a finite-dimensional Lie algebra, $J$ be a Lie-orthogonal operator on~$L$ and all its eigenvalues differ from~$\pm 1$. 
Then $L$ is a solvable Lie algebra of solvability degree not greater than two. 
\end{theorem}

\begin{proof}
Theorem's hypotheses imply that $L$ is a sum of the ideals~$I_\lambda$, where $\lambda$ runs through nonzero eigenvalues of the operator~$J$. 
(See the definition of~$I_\lambda$ in Lemma~\ref{LemmaOnIdealsRelatedToRootSubsp}.)
Note that this sum is not direct as the intersection of every pair of these ideals coincides with the center~$Z$ of the algebra~$L$. 
For each~$\lambda$ the ideal~$I_\lambda$ is a solvable ideal of solvability degree not greater than two. These ideals commutate with each other.
Therefore, the algebra~$L$ is also solvable of solvability degree not greater than two.
\end{proof}

\subsection{Lie-orthogonal automorphisms}

In this section we study the intersection of the automorphism group of a Lie algebra~$L$ and the set of its Lie-orthogonal operators.

\begin{lemma}\label{LemmaOnLieOrtAut1}
Let $J$ be a Lie-orthogonal automorphism on~$L$ and $L'=[L,L]$ denote the derived algebra of~$L$. 
Then $\ker (J-\Id_L)\supset L'$ and $[\mathrm{im} (J-\Id_L), L']=\{0\}$. 
\end{lemma}

\begin{proof}
For any elements $x$ and $y$ from $L$ we have $J[x,y]=[Jx,Jy]=[x,y]$. This implies that $(J-\Id_L)[x,y]=0$.
Therefore, $\ker (J-\Id_L)\supset L'$.
Analogously, for any elements $x$, $y$ and $z$ from $L$ we get $[[x,y],Jz]=[[Jx,Jy],Jz]=[J[x,y],Jz]=[[x,y],z]$, 
i.e., $[[x,y],Jz-z]=0$. This means that $[\mathrm{im} (J-\Id_L), L']=\{0\}$.
\end{proof}

\begin{corollary}
Any Lie-orthogonal automorphism on a perfect algebra acts identically.
\end{corollary}

Recall that a Lie algebra is called perfect if it coincides with its derived algebra.

It is obvious that
Lie-orthogonal automorphisms on a Lie algebra can be characterized as automorphisms which act identically on the corresponding derived algebra
or as Lie-orthogonal operators satisfying the same condition.

\begin{lemma}
Let a Lie algebra~$L$ be finite dimensional and~$J$ be a Lie-orthogonal automorphism on~$L$.
Then $\mathrm{im} (J-\Id_L)\subset R$, where $R$~is the radical of~$L$.
\end{lemma}

\begin{proof}
We fix arbitrary elements $x$, $y$ and $z$ from~$L$.
Using the invariance of the Killing form $K(x, y)$ of the algebra~$L$ 
with respect to each automorphism of~$L$, we derive that $K([Jx,Jy],Jz)=K([x,y],z)$.
At the same time, Lie orthogonality of~$J$ implies that $K([Jx,Jy],Jz)=K([x,y],Jz)$.
Therefore, $K([x,y],Jz-z)=0$, i.e., the image $\mathrm{im} (J-\Id_L)$ is contained 
in the orthogonal complement to the derived algebra~$L'$ with respect to the Killing form~$K$.
As well known, this complement coincides with the radical~$R$.
\end{proof}

Suppose that the underlying field of~$L$ is algebraically closed.
Given a number~$\mu$, let $I_\mu$ denote the ideal of the algebra~$L$ that was introduced in Lemma~\ref{LemmaOnIdealsRelatedToRootSubsp}.

\begin{corollary}\label{CorollaryOnLieOrtAutIdealsNilpotency}
Let $J$ be a Lie-orthogonal automorphism of a finite-dimensional Lie algebra~$L$.
Then for any eigenvalue~$\mu$ of~$J$, which is not equal to one, the corresponding
ideal $I_\mu$ is a nilpotent ideal of nilpotence degree not greater than two.
\end{corollary}

\begin{proof}
As $I_\mu$ is an ideal of~$L$, the commutator $[I_\mu, I_\mu]$ is contained in $I_\mu$.
Moreover, $[I_\mu, I_\mu]\subset L'$ by the definition of $L'$ and $L'\subset I_1$ in view of Lemma~\ref{LemmaOnLieOrtAut1}.
Hence $[I_\mu, I_\mu]\subset I_\mu\cap I_1=Z$, where~$Z$ is the center of~$L$, 
i.e., the nilpotence degree of~$I_\mu$ is not greater than two.
\end{proof}

\begin{corollary}
All eigenvalues of any Lie-orthogonal automorphism of a centerless finite-dimensional Lie algebra are equal to one.
\end{corollary}

\begin{proof}
Let $J$ be a Lie-orthogonal automorphism on a centerless finite-dimensional Lie algebra~$L$.
As an automorphism, the operator~$J$ has no zero eigenvalues.
We fix a number $\nu$ that is not zero or one.
Consider the corresponding ideal~$I_\nu$.
Analogously to the previous corollary, $[I_\nu, I_\nu]\subset Z$.
As $[I_\nu, I_\mu]=\{0\}$ for any eigenvalue $\mu$ that is not equal to $\nu$ or $\nu^{-1}$, 
and the algebra~$L$ is a sum of the ideals~$I_\mu$, where the subscript $\mu$ 
runs through the set of eigenvalues of the operator~$J$, the ideal~$I_\nu$ is contained in~$Z$.
Therefore, $I_\nu=\{0\}$, i.e., $\nu$ is not an eigenvalue of~$J$.
\end{proof}

\begin{corollary}
Let $J$ be a Lie-orthogonal automorphism on a finite-dimensional Lie algebra~$L$.
Then $L$ can be represented as a sum of two ideals $I_1+\hat I$, 
where the restriction of $J$ on $I_1$ has only unit eigenvalues and
$\hat I$ is a nilpotent ideal of nilpotence degree not greater than two.
\end{corollary}

\section{Lie-orthogonal operators on Lie algebras with nonzero\\ Levi factor}\label{SectionOnLieOrtOpsOnNonzeroLeviFactor}

\subsection{Semi-simple algebras}

The rigid structure of semi-simple Lie algebras imposes strong restrictions on the corresponding sets of Lie-orthogonal operators.

\begin{lemma}\label{LemmaOnAnnPoly}
The polynomial $t^2-1$ annihilates any Lie-orthogonal operator on a semi-simple finite-dimensional Lie algebra.
\end{lemma}

\begin{proof}
Let $K(x, y)$ be the Killing form of a semi-simple Lie algebra~$L$ and $J$ be a Lie-orthogonal operator on~$L$.
By the definition of Lie orthogonality and in view of the associativity of the Killing form with respect to Lie bracket \cite{Humphreys2000}, 
for arbitrary $x$, $y$ and $z$ from $L$ we have the following equalities:
\begin{gather*}
K([Jx, Jy], Jz)=K(Jx, [Jy, Jz])=K(Jx, [y, z])=K([Jx, y], z),\\[.5ex]
K([Jx, Jy], Jz)=-K([Jy, Jx], Jz)=-K(Jy, [Jx, Jz])=\\ 
=-K(Jy, [x, z])=-K([Jy, x], z)=K([x, Jy], z).
\end{gather*}
Therefore, $K([Jx, y], z)=K([x, Jy], z)$, i.e., $K([Jx, y]-[x, Jy], z)=0$.
As the element~$z$ is arbitrary and the form~$K$ is nondegenerate as the Killing form of a semi-simple algebra, we have
$[Jx, y]-[x, Jy]=0$, i.e., $[Jx, y]=[x, Jy]$ for any~$x$ and~$y$ from~$L$.
Combining this property with the definition of Lie-orthogonal operators implies
$[x, y]=[Jx, Jy]=[x, J^2y]$, whence $[x, J^2y-y]=0$. 
As the algebra~$L$ is centerless and the element~$x$ is arbitrary, we get $J^2y-y=0$ for all~$y$ from~$L$.
This means that $J^2-\Id_L=0$. In other words, the polynomial $t^2-1$ annuls~$J$.
\end{proof}

It follows from Lemma~\ref{LemmaOnAnnPoly} that each eigenvalue of the operator~$J$ is equal to either~$1$ or~$-1$.
This agrees with Corollaries~\ref{CorollaryEigenJOnSimpleAlg} and~\ref{CorollaryEigenJOnSemisimpleAlg}.
At the same time, the statement of this lemma is much stronger: 
It guarantees the existence of a basis of~$L$, which is formed by eigenvectors of~$J$.
Taking into account Corollary~\ref{CorollaryEigenJOnSemisimpleAlg}, we obtain the following assertions:

\begin{theorem}\label{TheoremOnLieOrtOpsOfSimpleAlg}
Lie-orthogonal operators on any simple finite-dimensional Lie algebra are exhausted by the trivial operators $\Id_L$ and~$-\Id_L$.
\end{theorem}

\begin{theorem}\label{TheoremOnLieOrtOpsOfSemisimpleAlg}
Let $L$ be a semi-simple finite-dimensional Lie algebra and $L=L_1\oplus \dots \oplus L_k$ be its decomposition into simple components.
Then any Lie-orthogonal operator $J$ on~$L$ can be represented in the form
\[
J=J_1\oplus \dots \oplus J_k,
\]
where for any $i=1,\dots,k$ we have either $J_i=\Id_{L_i}$ or $J_i=-\Id_{L_i}$.
\end{theorem}

In other words, any Lie-orthogonal operator~$J$ leads to the partition of the semi-simple algebra~$L$ into two ideals.
The operator~$J$ acts on one of the ideals as~$\Id$ and on the other as~$-\Id$.

Theorem~\ref{TheoremOnLieOrtOpsOfSemisimpleAlg} implies that 
$(0,1,1)$-differentiations of a semi-simple finite-dimensional Lie algebra are exhausted by the zero differentiation. 
This agrees with Lemma~6.1 from~\cite{Leger&Luks2000}.

\subsection{Reductive algebras}

Combining Lemma~\ref{LemmaOnOpOnSumIdealsEqualsSumOp} with Theorem~\ref{TheoremOnLieOrtOpsOfSemisimpleAlg}, we can 
directly obtain the complete description of Lie-orthogonal operators on reductive Lie algebras.

\begin{corollary}\label{CorollaryOnLieOrtOpsOfReductiveAlg}
Let $L$ be a finite-dimensional reductive Lie algebra and $L=Z\oplus L_1\oplus \dots \oplus L_k$ be its decomposition into the center and simple components.
An operator~$J$ is Lie-orthogonal on~$L$ if and only if it can be represented in the form
\[
J=J_0\oplus J_1\oplus \dots \oplus J_k+J_Z,
\]
where for each $i=1,\dots,k$ we have either $J_i=\Id_{L_i}$ or $J_i=-\Id_{L_i}$, $J_0$ is the identically zero operator on~$Z$ 
and $J_Z$ is an arbitrary operator on~$L$ whose image is contained in the center~$Z$.
\end{corollary}

\subsection{Levi factor and generalized eigenspaces of operators}\label{SectionOnLeviFactorAndLieOrtOps}

Suppose that the underlying field is of characteristic~$0$ and algebraically closed.
In view of Lemma~\ref{LemmaOnIdealsRelatedToRootSubsp}, a Lie algebra~$L$ possessing a Lie-orthogonal operator~$J$
can be represented as the sum of the ideals~$I_\lambda$, each of which has either the form $I_\lambda=L_\lambda\oplus L_{\lambda^{-1}}\oplus Z$ 
if $\lambda\notin\{\pm 1, 0\}$ or the form $I_\lambda=L_\lambda\oplus Z$ if $\lambda=\pm 1$.
Here~$\lambda$ runs through the set of nonzero eigenvalues of the operator~$J$. If~$\lambda$ and~$\lambda^{-1}$ are simultaneously eigenvalues of~$J$, 
we omit the ideal $I_{\lambda^{-1}}$ from the sum in order to avoid a repetition. This sum is not direct in the general case 
since each ideal~$I_\lambda$ contains the center~$Z$ of~$L$.
Nevertheless, these ideals can be separately studied, up to the equivalence, 
with respect to the restrictions imposed on their structure by the operator~$J$.
This is possible because these ideals are invariant subspaces of the operator~$J$ 
and additionally $[I_\lambda, I_\mu]=0$ when $\lambda\ne\mu$ and $\mu\lambda\ne 1$.

Each ideal~$I_\lambda$, where $\lambda\ne\pm 1$, is solvable and hence is contained in the radical~$R$ of the algebra~$L$.
Therefore, only the ideals $I_1$ and~$I_{-1}$ may have nonzero intersections with a Levi factor of~$L$.
Due to the involution $J\to -J$ on the set of Lie-orthogonal operators, it suffices to study the ideal~$I_1$.

The intersections of the radical~$R$ and a Levi factor~$S$ of the algebra~$L$ with~$I_1$ are
the radical~$R_1$ and a Levi factor~$S_1$ of~$I_1$, respectively, see, e.g., Corollary~4 from~\cite{Bourbaki1975}. 
Note that $S=S_1\oplus S_{-1}$, where $S_{-1}=S\cap I_{-1}$.
Lemma~\ref{LemmaOnInvarienceOfRad} implies that the radical~$R_1$ is invariant with respect to the operator~$J$ 
and hence with respect to the restriction of~$J$ on~$I_1$, too.
The factor-algebra of the ideal~$I_1$ with respect to its radical~$R_1$ is isomorphic to the subalgebra~$S_1$ and is obviously semi-simple.
As the radical~$R_1$ is invariant with respect to~$J$, 
the factorization of the restriction~$J_1$ of the operator~$J$ on~$I_1$ with respect to~$R_1$ gives 
the well-defined operator~$J_1/R_1$ on the factor-algebra $I_1/R_1$ with a single eigenvalue equal to~1.
In view of Theorem~\ref{TheoremOnLieOrtOpsOfSemisimpleAlg}, the factorized operator~$J_1/R_1$ acts identically.
Therefore, the image of the operator~$J_1-\Id_{I_1}$ is contained in~$R_1$, and this operator is nilpotent by the definition of~$I_1$ up to the equivalence relation.
We sum up the obtained result as the following assertion:

\begin{proposition}\label{PropositionOnLeviFactorsAndJEqIdPlusNil}
Let $J$ be a Lie-orthogonal operator on a finite-dimensional Lie algebra~$L$ with the eigenvalue $\lambda=1$ and $I_1$ be the ideal corresponding to this eigenvalue.
Then the restriction~$J_1$ of the operator~$J$ on the ideal~$I_1$ can be represented, 
up to equivalence of Lie-orthogonal operators on~$I_1$, in the form $J_1=\Id_{I_1}+N$, 
where $N$ is a nilpotent operator on~$I_1$ and the image of~$N$ is contained in the radical~$R_1$ of the ideal~$I_1$.
\end{proposition}

\section{Direct calculation of Lie-orthogonal operators}\label{SectionOnDirectCalcOfLieOrtOps}

For some classes of Lie algebras of simple structure, we can completely describe their Lie-orthogonal operators
using only the definition of Lie-orthogonality and commutation relations in the canonical bases of these algebras.
Such computation is often a necessary step in the study of Lie-orthogonal operators, 
giving a base for making conjectures about general properties of such operators.

\subsection{Special linear algebras}\label{SectionOnDirectCalcOfLieOrtOpsForSL}

For convenience we calculate Lie-orthogonal operators on $\mathrm{gl}_n$ instead of~$\mathrm{sl}_n$.
This is possible because the algebra~$\mathrm{gl}_n$ is a central extension of the algebra~$\mathrm{sl}_n$: $\mathrm{gl}_n=\mathrm{sl}_n\oplus\langle E_n\rangle$, 
where~$E_n$ is the unit matrix of size~$n$, which generates the center~$\langle E_n\rangle$ of the algebra~$\mathrm{gl}_n$.
Hence in view of Lemma~\ref{LemmaOnIndeterminacyOfLieOrtOpsOnCenter2} the operator~$J$ is Lie-orthogonal on the algebra~$\mathrm{gl}_n$
if and only if the operator $PJ|_{\mathrm{sl}_n}$ is Lie-orthogonal on the algebra~$\mathrm{sl}_n$. 
Here $P$ denotes the projection operator from~$\mathrm{gl}_n$ onto $\mathrm{sl}_n$, which is associated with the above decomposition of~$\mathrm{gl}_n$.

As the Lie algebra~$\mathrm{sl}_n$ is simple for any~$n$ and hence any Lie-orthogonal operator on it is nondegenerate, 
it suffices to consider only the corresponding nondegenerate Lie-orthogonal operators on~$\mathrm{gl}_n$.

Let $J$ be such an operator. 
Then the condition of Lie orthogonality can be represented in the form $[Jx, y]=[x, J^{-1}y]$.
We write it down for the matrix units $x=E^{ij}$ and $y=E^{kl}$ of the algebra~$\mathrm{gl}_n$ for fixed values of the indices~$i$, $j$, $k$ and $l$ from $\{1,\dots,n\}$
using the notations $JE^{ij}=A^{ij}=A^{ij}_{pq}E^{pq}$ and $J^{-1}E^{kl}=B^{kl}=B^{kl}_{pq}E^{pq}$.
Here and in what follows we assume summation from~$1$ to~$n$ by the repeated indices $p$ and $q$. 
Therefore, the Lie-orthogonality condition for the matrix units is $[A^{ij}, E^{kl}]-[E^{ij}, B^{kl}]=[A^{ij}, E^{kl}]+[B^{kl}, E^{ij}]=0$.
After expanding $A^{ij}$ and $B^{kl}$ by the basis, we obtain
\[
A^{ij}_{pk}E^{pl}-A^{ij}_{lq}E^{kq}+B^{kl}_{pi}E^{pj}-B^{kl}_{jq}E^{iq}=0.
\]
Considering different possibilities for the values of the indices~$i$, $j$, $k$ and $l$, 
we collect the coefficients of the basis elements in the last equality and equate these coefficients to zero.
As a result, we derive the following system:
\begin{gather*}
k\ne i, \ l\ne j \\
E^{pl}, p\ne i, k\colon \ A^{ij}_{pk}=0,\quad E^{pj}, p\ne i, k\colon \  B^{kl}_{pi}=0,\\
E^{kq}, q\ne j, l\colon \ A^{ij}_{lq}=0,\quad E^{iq}, q\ne j, l\colon \  B^{kl}_{jq}=0,\\
E^{il} \colon \ A^{ij}_{ik}=B^{kl}_{jl},\quad E^{kl} \colon \ A^{ij}_{kk}=A^{ij}_{ll},\\
E^{kj} \colon \ A^{ij}_{lj}=B^{kl}_{ki},\quad E^{ij} \colon \ B^{kl}_{ii}=B^{kl}_{jj};
\\[1ex]
k=i, \ l\ne j \\
E^{pl}, p\ne i\colon \ A^{ij}_{pi}=0,\quad E^{pj}, p\ne i\colon \  B^{kl}_{pk}=0,\\
E^{kq}, q\ne j, l\colon \ A^{ij}_{lq}+B^{kl}_{jq}=0,\\
E^{il} \colon \ A^{ij}_{ii}=A^{ij}_{ll}+B^{kl}_{jl},\quad 
E^{ij} \colon \ B^{kl}_{kk}=A^{ij}_{lj}+B^{kl}_{jj};
\\[1ex]
k\ne i, \ l=j \\
E^{kq}, q\ne j\colon \ A^{ij}_{jq}=0,\quad E^{iq}, q\ne j\colon \  B^{kl}_{lq}=0,\\
E^{pj}, p\ne k, i\colon \ A^{ij}_{pk}+B^{kl}_{pi}=0,\\
E^{kj} \colon \ A^{ij}_{jj}=A^{ij}_{kk}+B^{kl}_{ki},\quad 
E^{ij} \colon \ B^{kl}_{ll}=A^{ij}_{ik}+B^{kl}_{ii};
\\[1ex]
k=i, \ l=j \\
E^{pj}, p\ne i\colon \ A^{ij}_{pi}+B^{kl}_{pk}=0,\quad
E^{iq}, q\ne j\colon \ A^{ij}_{jq}+B^{kl}_{lq}=0,\\
E^{ij} \colon \ A^{ij}_{ii}+B^{kl}_{kk}=A^{ij}_{jj}+B^{kl}_{ll}.
\end{gather*}

As the system is symmetric with respect to~$A^{ij}$ and~$B^{kl}$, 
it suffices to study only the constraints, imposed on~$A^{ij}$.
The system implies that $A^{ij}_{pq}=0$ if $p\ne q$ and $(p, q)\ne (i, j)$; $A^{ij}_{kk}=A^{ij}_{ll}$ if $i\ne j$;
and $A^{ii}_{kk}=A^{ii}_{ll}$ if $k, l\ne i$. This means that $A^{ij}=\lambda_{ij}E^{ij}+\kappa_{ij}E_n$, where $\lambda_{ij}$ and $\kappa_{ij}$ are constants.
Since the values of the operator~$J$ on the basis elements are defined up to adding elements of the center, the constant~$\kappa_{ij}$ can be assumed zero,
i.e., $A^{ij}=JE^{ij}=\lambda_{ij}E^{ij}$.
In other words, each basis element $E^{ij}$ is an eigenvector of the operator~$J$. 
Then for an arbitrary index triple $(i,j,k)$ such that $(i, k)\ne (k, j)$ we obtain
\[[E^{ik}, E^{kj}]=[JE^{ik}, JE^{kj}]=\lambda_{ik}\lambda_{kj}[E^{ik}, E^{kj}].\] 
Hence $\lambda_{ik}\lambda_{kj}=1$ since $[E^{ik}, E^{kj}]\ne 0$.
Therefore, either all $\lambda_{ij}$ equal $1$ or they all equal $-1$.

As a result, we obtain Theorem~\ref{TheoremOnLieOrtOpsOfSimpleAlg} for the particular case $L=\mathrm{sl}_n$.

\subsection{Heisenberg algebras}

Consider the Heisenberg algebra $\mathrm h_n$ for a fixed value of~$n$.
This is a nilpotent Lie algebra of dimension~$2n+1$ and nilpotency degree 2 with one-dimensional center.
We fix a basis $\{e, p_1,\dots, p_n, q_1,\dots, q_n\}$ of the algebra $\mathrm h_n$, in which the nonzero commutation relations take the canonical form
\[
[p_j, q_j]=e
\]
and hence the center of $\mathrm h_n$ is $Z=\langle e\rangle$.
Here and in what follows the indices $i$, $j$ and $k$ run from~$1$ to~$n$.

Let $J$ be a Lie-orthogonal operator on~$\mathrm h_n$ and $\hat J$ be its essential part, i.e.,
$\hat J=(\mathrm PJ)|_{\hat{\mathrm h}}$, where $\hat{\mathrm h}=\langle p_i, q_j\rangle$ and $P$ is the projection operator onto~$\hat{\mathrm h}$ 
in the representation $\mathrm h_n=Z\dotplus \hat{\mathrm h}$.
We denote the matrix of the operator~$\hat J$ in the canonical basis $\{p_1,\dots, p_n, q_1,\dots, q_n\}$ by the same symbol as the operator.
In accordance with the basis partition into $\{p_1, \dots, p_n\}$ and $\{q_1, \dots, q_n\}$, we split the matrix $\hat J$ into blocks:
\[
\hat J=\left(\begin{array}{cc}A&B\\C&D\end{array}\right), 
\]
where $A=(a_{ij})$, $B=(b_{ij})$, $C=(c_{ij})$, $D=(d_{ij})$ are $n\times n$ matrices. Then $\hat Jp_j=p_ia_{ij}+q_ic_{ij}$ and $\hat Jq_k=p_ib_{ik}+q_id_{ik}$. Here and in what follows we assume summation with respect to the repeated index $i$. In view of the definition of Lie-orthogonal operators, we have
\begin{gather*}
[Jp_j, Jq_k]=(a_{ij}d_{ik}-c_{ij}b_{ik})e=\delta_{ik}e, \\
[Jp_j, Jp_k]=(a_{ij}c_{ik}-c_{ij}a_{ik})e=0,\\
[Jq_j, Jq_k]=(b_{ij}d_{ik}-d_{ij}b_{ik})e=0,
\end{gather*}
where $\delta_{jk}$ is the Kronecker delta.
Using the matrix notation these equalities are written in the form \[
A^{\mathrm T}D-C^{\mathrm T}B=E, \quad
A^{\mathrm T}C-C^{\mathrm T}A=0, \quad
B^{\mathrm T}D-D^{\mathrm T}B=0,
\]
where $E\in M_n$ is the unit matrix.
This means that $\hat J\in\mathrm{Sp}_{2n}$. 
Indeed, let \[S=\begin{pmatrix}0&-E\\E&0\end{pmatrix}\] be the matrix of the canonical symplectic form.
Then
\[
\hat J^{\mathrm T}S\hat J=
\begin{pmatrix}
C^{\mathrm T}A-A^{\mathrm T}C&C^{\mathrm T}B-A^{\mathrm T}D\\
D^{\mathrm T}A-B^{\mathrm T}C&D^{\mathrm T}B-B^{\mathrm T}D
\end{pmatrix}
=S.
\]

As a result, the following assertion is true:

\begin{theorem}
An operator is Lie-orthogonal on the Heisenberg algebra~$\mathrm h_n$ if and only if in the canonical basis
$\{e, p_1,\dots, p_n, q_1,\dots, q_n\}$ its matrix takes a form
\[\left(\begin{array}{cc}r&R\\0&\hat J\end{array}\right),\] 
where
$r$~is an arbitrary scalar, $R$~is an arbitrary matrix from $M_{1,2n}$,
$0$ denotes the zero matrix from $M_{2n,1}$ and $\hat J$ is an arbitrary matrix from~$\mathrm{Sp}_{2n}$.
\end{theorem}

\subsection{Almost Abelian algebras}

A Lie algebra is called almost Abelian if it contains an Abelian ideal of codimension 1.
Each almost Abelian algebra is solvable of solvability degree not greater than two.

Let $L$ be an almost Abelian but not Abelian algebra of dimension~$n$.
We choose a basis in~$L$ such that $e_1$, \dots, $e_{n-1}$ form a basis of an $(n-1)$-dimensional Abelian ideal~$I$ contained in~$L$.
Then the equations $[e_n, e_j]=\sum_{i=1}^{n-1}a_{ij}e_i$ exhaust all nonzero commutation relations of the algebra~$L$.
Here and in what follows the indices $i$, $j$, $i'$ and~$j'$ run from~$1$ to~$n-1$.
We assume summation with respect to repeated indices.
Thus, the nonzero matrix $A=(a_{ij})\in M_{n-1}$ completely defines the almost Abelian algebra~$L$. 
Note that the center of the algebra~$L$ coincides with the kernel of the operator~$A$, which acts on~$I$ and is defined by the matrix~$A$.  
Almost Abelian Lie algebras~$L$ and~$\tilde L$ are isomorphic if the corresponding matrices~$A$ and~$\tilde A$ are similar up to a constant multiplier, 
i.e., there exist a nonzero constant~$\mu$ and a nondegenerate matrix~$S\in M_{n-1}$ such that $\tilde A=\mu S^{-1}AS$.

\begin{proposition}
Let $L$ be an $n$-dimensional almost Abelian but not Abelian Lie algebra with a fixed basis such that
$\langle e_1,\dots, e_{n-1}\rangle$ is an Abelian ideal of~$L$ and $\langle e_1,\dots, e_m\rangle=Z$, where $Z$ is the center of~$L$ and $m=\dim Z<n-1$. 
In the chosen basis the nonzero commutation relations take the form $[e_n, e_j]=a_{ij}e_i$ and thus define the matrix $A=(a_{ij})$, 
where the first~$m$ columns of~$A$ are zero and the rank of~$A$ equals~$n-m$.
An operator~$J$ on~$L$ is Lie-orthogonal if and only if its matrix in the basis $\{e_1, \dots, e_n\}$ has one of the following forms
depending on $\dim Z$:

1. $m=\dim Z<n-2$:
\[
\left(
\begin{array}{ccc}
B_0&B_1&B_2\\
0&\mu E&B_3\\
0&0&\mu^{-1}
\end{array}
\right),
\]
where $B_0$, $B_1$, $B_2$ and~$B_3$ are arbitrary~$m\times m$, $m\times (n-m-1)$, $m\times 1$ and $(n-m-1)\times 1$ matrices, respectively, 
$\mu$ is an arbitrary nonzero constant, $E$~is the unit matrix of size~$n-m-1$ and zeros denote zero matrices of appropriate sizes. 

2. $m=\dim Z=n-2$, i.e., the algebra~$L$ is isomorphic to either $(n-2)\mathfrak g_1\oplus\mathfrak g_2$ or $(n-3)\mathfrak g_1\oplus\mathrm h_3$:
\[
\left(
\begin{array}{cc}
B_0&B_1\\
0&C
\end{array}
\right),
\]
where $B_0$ and $B_1$ are arbitrary $(n-2)\times(n-2)$ and $(n-2)\times 2$ matrices, respectively, $0$ is the zero $2\times 2n$ matrix, 
and $C$ is an arbitrary $2\times 2$ matrix with the determinant equal to one, i.e., $C\in \mathrm{SL}_2$.
\end{proposition}

\begin{proof}
We split the matrix of the operator~$J$ into the blocks with respect to the representation of the algebra~$L$ as the direct sum of the spaces $I=\langle e_1, \dots e_{n-1}\rangle$ and~$\langle e_n \rangle$:
\[
J=
\left(
\begin{array}{cc}
J_{II}&J_{In}\\
J_{nI}&J_{nn}\\
\end{array}
\right).
\]
In view of the definition of Lie-orthogonal operators, we have
\begin{gather*}
0=[e_i, e_j]=[Je_i, Je_j]=[J_{ki}e_k, J_{lj}e_l]=(J_{ni}J_{i'j}-J_{nj}J_{i'i})a_{j'i'}e_{j'},\\
a_{j'j}e_{j'}=[e_n, e_j]=[Je_n, Je_j]=[J_{kn}e_k, J_{lj}e_l]=(J_{nn}J_{i'j}-J_{nj}J_{i'n})a_{j'i'}e_{j'},
\end{gather*}
where indices $k$ and~$l$ run from~$1$ to~$n$, or, in the matrix notation,
\begin{gather}\label{EqAlmostAbAlg1}
J_{ni}(AJ_{II})_{j'j}=J_{nj}(AJ_{II})_{j'i},\\\label{EqAlmostAbAlg2}
J_{nn}(AJ_{II})_{j'j}-J_{nj}(AJ_{jn})_{j'}=a_{j'j}.
\end{gather}
The system~\eqref{EqAlmostAbAlg1} implies that either $J_{nj}=0$ for each~$j$ or there is $j_0$ such that $J_{nj_0}\ne 0$.
We consider these alternatives separately.

1. Suppose that $J_{nj}=0$ for each~$j$. Then it follows from~\eqref{EqAlmostAbAlg2} that $J_{nn}AJ_{II}=A$. As the matrix~$A$ is nonzero, we have $J_{nn}\ne 0$.
We denote~$1/J_{nn}$ by~$\mu$. The system~\eqref{EqAlmostAbAlg2} takes the form $A(J_{II}-\mu E)=0$. 
Therefore, $\ker A$, which coincides with the center~$Z$ of~$L$, contains the image of the operator~$J_{II}-\mu E$.
It proves the first case of the theorem.

2. Suppose that there is such~$j_0$ that $J_{nj_0}\ne 0$.
Then all tuples $((AJ_{II})_{1j}, \dots, (AJ_{II})_{n-1,j}, J_{nj})$ are proportional, i.e.,
$J_{nj}=\lambda_jJ_{nj_0}$ and $(AJ_{II})_{ij}=\lambda_j(AJ_{II})_{ij_0}$ for some constants~$\lambda_j$.
Substituting the obtained expressions into the system~\eqref{EqAlmostAbAlg2}, we obtain that $a_{ij}=\kappa_i\lambda_j$, 
where $\kappa_i=J_{nn}(AJ_{II})_{ij_0}-J_{nj_0}(AJ_{In})_i$. As $A\ne 0$, this means that the rank of the matrix~$A$ equals one.
Up to the choice of basis, it can be assumed that either $A=E^{n-1,n-1}$ or $A=E^{n-2, n-1}$.
In the term of~$\kappa$ and~$\lambda$ we have that $\lambda_1=\dots=\lambda_{n-2}=0$ and hence $j_0=n-1$ and $\lambda_{n-1}=1$.
Moreover, either $\kappa_{n-1}$ or $\kappa_{n-2}$ is equal to one, respectively, and all the other $\kappa$'s are zero.
The equation $J_{nj}=\lambda_jJ_{n, n-1}$ implies that $J_{nj}=0$ for each~$j$ from~$1$ to~$n-2$.
Analogously, the equation $(AJ_{II})_{ij}=\lambda_j(AJ_{II})_{i,n-1}$ for either $i=n-1$ or $i=n-2$ respectively gives that $J_{n-1,j}=0$ for each~$j$ from~$1$ to~$n-2$.
One more equation $J_{n-1,n-1}J_{nn}-J_{n-1,n}J_{n,n-1}=1$ follows from the definition of~$\kappa$.
The above equations form the complete system on the entries of the matrix of~$J$. This leads to the second case of the proposition.
\end{proof}

\subsection{Solvable algebras with nilradicals of minimal dimension}

It was proven by Mubarakzyanov \cite[Theorem 5]{Mubarakzyanov1963} that the dimension of the nilradical~$N$ of an $n$-dimensional solvable Lie algebra~$L$ over a field of characteristic zero is not less than~$n/2$.
If $\dim N=n/2$ in the case of even~$n$ then the algebra~$L$ is the direct sum of
$n/2$~copies of the two-dimensional non-Abelian Lie algebra~$\mathfrak g_2$, $L=(n/2)\mathfrak g_2$ \cite[Theorem 6]{Mubarakzyanov1963}.
Theorem~7 from the same paper~\cite{Mubarakzyanov1963} implies that in the case of odd~$n$ the minimal dimension of the nilradical equals $[n/2]+1=(n+1)/2$ 
and the algebra~$L$ with the nilradical of this dimension can be decomposed into the direct sum
of a single one-dimensional (Abelian) Lie algebra~$\mathfrak g_1$ and $n/2$~copies of the algebra $\mathfrak g_2$, i.e., $L=\mathfrak g_1\oplus[n/2]\mathfrak g_2$. 
In view of Lemma~\ref{LemmaOnOpOnSumIdealsEqualsSumOp} any Lie-orthogonal operator~$J$ on the algebra~$L$ is decomposed into the direct sum of Lie-orthogonal operators
on the above components of the algebra~$L$. 
In the case of odd dimension the algebra~$L$ has the nonzero center $Z=\mathfrak g_1$ and so the decomposition is up to the Lie-orthogonal operators equivalence.
As the algebra~$\mathfrak g_2$ is centerless, the Lie-orthogonal operators on this algebra form a group.
It can be directly calculated (see also~\cite[Example 1]{Bilun&Maksimenko&Petravchuk2011}) that this group coincides with~$\mathrm{SL}_2$.
After taking into account Lemma~\ref{LemmaOnIndeterminacyOfLieOrtOpsOnCenter2}, we obtain the following \mbox{assertion}.
\looseness=1

\begin{proposition}
Suppose that the nilradical of an $n$-dimensional solvable Lie algebra~$L$ is of minimal dimension which is equal to $[n+1]/2$.
If the dimension~$n$ is even, the algebra~ $L$ is isomorphic to the algebra $[n/2]\mathfrak g_2$ and 
the Lie-orthogonal operators on~$L$ form a group isomorphic to the direct product of $n/2$ copies of the group~$\mathrm{SL}_2$.
If $n$ is odd, the algebra~$L$  is isomorphic to $\mathfrak g_1\oplus[n/2]\mathfrak g_2$ and each Lie-orthogonal operator~$J$ on~$L$ can be represented as
$J=J_0\oplus J_1\oplus\dots\oplus J_{[n/2]}+J_Z$. Here~$J_0$ is the zero operator on the center~$Z=\mathfrak g_1$ of~$L$,
$J_i$ is a Lie-orthogonal operator on the $i$th copy of $\mathfrak g_2$ (i.e., its matrix belongs to the group $\mathrm{SL}_2$) 
and $J_Z$ is an operator on~$L$ whose image is contained in~$Z=\mathfrak g_1$.
\end{proposition}

\subsection{Low-dimensional non-solvable complex Lie algebras}

Non-solvable complex Lie algebras of dimension~$n\leqslant 5$ are exhausted by algebras $\mathrm{sl}_2$ ($n=3$), 
$\mathrm{sl}_2\oplus\mathfrak g_1$ ($n=4$), 
$\mathrm{sl}_2\oplus2\mathfrak g_1$, $\mathrm{sl}_2\oplus\mathfrak g_2$ and $\mathrm{sl}_2\lsemioplus2\mathfrak g_1$ ($n=5$).

The algebra $L=\mathrm{sl}_2$ is simple, hence in view of Theorem~\ref{TheoremOnLieOrtOpsOfSimpleAlg} all Lie-orthogonal operators on this algebra are exhausted by the trivial ones, $\Id_L$ and~$-\Id_L$.
See also Subsection~\ref{SectionOnDirectCalcOfLieOrtOpsForSL}.
The algebras $\mathrm{sl}_2\oplus\mathfrak g_1$ and $\mathrm{sl}_2\oplus2\mathfrak g_1$ are reductive
and Lie-orthogonal operators on reductive algebras are completely described in Corollary~\ref{CorollaryOnLieOrtOpsOfReductiveAlg}.

The algebra $\mathrm{sl}_2\oplus\mathfrak g_2$ has zero center and is a direct sum of its ideals~$\mathrm{sl}_2$ and $\mathfrak g_2$.
In view of Lemma~\ref{LemmaOnOpOnSumIdealsEqualsSumOp}, any Lie-orthogonal operator on $\mathrm{sl}_2\oplus\mathfrak g_2$ is a direct sum of Lie-orthogonal operators on these ideals.
The sets of Lie-orthogonal operators on~$\mathrm{sl}_2$ and $\mathfrak g_2$ have already been described.
Therefore, the Lie-orthogonal operators on the algebra $\mathrm{sl}_2\oplus\mathfrak g_2$ form a group, which is isomorphic to $\mathbb Z_2\times\mathrm{SL}_2$.

The only remaining algebra is~$L=\mathrm{sl}_2\lsemioplus2\mathfrak g_1$.
The nonzero commutation relations of this algebra in the canonical basis are exhausted by the following:
\begin{gather*}
[e_1, e_2] = 2e_2,\quad
[e_1, e_3] =-2e_3,\quad
[e_2, e_3] =  e_1,\\[.5ex]
[e_1, e_4] =  e_4,\quad
[e_2, e_5] =  e_4,\quad
[e_3, e_4] =  e_5,\quad
[e_1, e_5] = -e_5,\quad
\end{gather*}
i.e., $e_1$, $e_2$ and $e_3$ form a basis of a Levi factor, which is isomorphic to~$\mathrm{sl}_2$, and $e_4$ and~$e_5$ form a basis of the Abelian radical~$R$ of this algebra.
Let $J$ be a Lie-orthogonal operator on~$L$.
The only ideal of the algebra~$L$ containing Levi factor is the very algebra~$L$. Also, note that the algebra~$L$ is centerless.
It follows from the content of Subsection~\ref{SectionOnLeviFactorAndLieOrtOps} that the operator~$J$ has the only eigenvalue, which equals either~$1$ or~$-1$.
Up to involution, we can assume that this eigenvalue is~$1$.
Then Proposition~\ref{PropositionOnLeviFactorsAndJEqIdPlusNil} implies that $J=\Id_L+N$, 
where $N$ is a nilpotent operator whose image is contained in the radical~$R$.
Therefore, the action of the operator~$J$ on the elements of the basis can be represented in the form:
\begin{gather*}
Je_i=e_i+a_ie_4+b_ie_5, \quad i=1,2,3, \\
Je_i=a_ie_4+b_ie_5, \quad i=4,5.
\end{gather*}
We apply the definition of Lie orthogonality to different pairs of basis elements:
\begin{gather*}
[Je_1, Je_2]=2e_2+a_2e_4-b_2e_5-b_1e_4=[e_1, e_2]=2e_2,\\
[Je_1, Je_3]=-2e_3+a_3e_4-b_3e_5-a_1e=[e_1, e_3]=-2e_3,\\
[Je_2, Je_3]=e_1+b_3e_4-a_2e_5=[e_2, e_3]=e_1,\\
[Je_1, Je_4]=a_4e_4-b_4e_5=[e_1, e_4]=e_4,\\
[Je_1, Je_5]=a_5e_4-b_5e_5=[e_1, e_5]=-e_5.
\end{gather*}
Thus we have the equations on the coefficients $b_2=0$, $a_2=b_1$, $a_3=0$, $a_1=-b_3$, $b_3=0$, $a_2=0$  (hence $a_i=b_i=0$, $i=1,2,3$), $a_4=1$, $b_4=0$, $a_5=0$ and $b_5=1$.
Summing up the calculation, we obtain the next proposition.

\begin{proposition}
Lie-orthogonal operators on the Lie algebra $L=\mathrm{sl}_2\lsemioplus2\mathfrak g_1$ are exhausted by the trivial operators~$\Id_L$ and~$-\Id_L$.
\end{proposition}

\section{Conclusion}

In this paper we have studied Lie-orthogonal operators on finite-dimensional Lie algebras over a field of characteristic 0.
One of the main directions of the study is the generalization of results obtained in~\cite{Bilun&Maksimenko&Petravchuk2011,Petravchuk&Bilun2003}
via removing the restrictions on algebras' centers and/or on operator nondegeneracy.
In particular, it has been proved that the center and all elements of the ascending central series of a Lie algebra 
are invariant with respect to any Lie-orthogonal operator on this algebra.
If a finite-dimensional Lie algebra over an algebraically closed field of characteristic 0
possesses a Lie-orthogonal operator whose eigenvalues are not equal to~$\pm 1$ 
then this algebra is solvable of solvability degree 2.
Proofs of a number of existing assertions have been simplified, 
including the lemma stating that the zero generalized eigenspace is contained in the center of the algebra,
the lemma on the commutation of pairs of generalized eigenspaces such that the product of the respective eigenvalues is not equal to~$1$ and 
the lemma on the invariance of an ideal the factor-algebra of which is centerless.

The natural equivalence relation of Lie-orthogonal operators has been introduced. It plays an important role in the entire consideration as many assertions have been formulated up to this equivalence.
In particular, it has been proved that the decomposition of a Lie algebra into the direct sum of the ideals implies the decomposition of any Lie-orthogonal operator on the algebra
into the direct sum of Lie-orthogonal operators on these ideals up to the above equivalence.

Perhaps, the most interesting result of the paper is the complete description of Lie-orthogonal operators on semi-simple Lie algebras.
The Lie-orthogonal operators on a simple Lie algebra are exhausted by the trivial ones, i.e., the identity operator and the minus identity operator.
Then any Lie-orthogonal operator on a semi-simple algebra can be represented as a direct sum of trivial operators on the simple components of the algebra.
This result allowed us to completely describe Lie-orthogonal operators on reductive Lie algebras and to derive certain properties of
Lie-orthogonal operators on algebras with nonzero Levi factors.
Though its proof looks simple, it significantly improves much more complicated Theorem~2 from~\cite{Bilun&Maksimenko&Petravchuk2011} 
which only gives a preliminary description of the group of Lie-orthogonal operators on a semi-simple algebra.

For some classes of Lie algebras, the respective sets of Lie-orthogonal operators have been found via direct calculations.
The list of such classes includes the special linear algebras, the Heisenberg algebras, the almost Abelian algebras and all non-solvable algebras of dimension not higher than five.
The calculation for the special linear algebras plays the test role for applying the basis approach.
The introduced notion of equivalence allowed us to briefly formulate the assertion about Lie-orthogonal operators on a Heisenberg algebra:
the equivalence classes of these operators form the symplectic group of the appropriate dimension (the algebra dimension minus one).
The study of Lie-orthogonal operators on almost Abelian algebras in an important step to the complete description of Lie-orthogonal operators
on low-dimensional algebras since almost Abelian algebras constitute a considerable portion of low-dimensional algebras.
The description of Lie-orthogonal operators on low-dimensional algebras gives necessary material for 
suggesting conjectures on Lie-orthogonal operators with non-trivial Levi--Maltsev decomposition.

Using the procedure of the algebraic closure it is possible to extend results obtained for algebraically closed fields
to fields which are not algebraically closed, although this extension may be nontrivial and hence requires a further investigation.

\subsection*{Acknowledgements}

The author is grateful to Prof.\ Anatoliy Petravchuk for posing of the problem about study of Lie-orthogonal operators, useful discussions and interesting comments.
It is our great pleasure to thank the referees for many useful remarks and suggestions that have considerably improved the paper.
The research was supported by the Austrian Science Fund (FWF), project P20632.

\end{document}